\documentclass[12pt]{amsart}
\usepackage{amsmath,amsthm}
\usepackage[dvips]{graphicx}

\numberwithin{equation}{section}
\newtheorem{thm}[equation]{Theorem}
\newtheorem{cor}[equation]{Corollary}
\newtheorem{lm}[equation]{Lemma}
\newtheorem{prp}[equation]{Proposition}

\theoremstyle{definition}
\newtheorem{df}[equation]{Definition}
\newtheorem{exmp}[equation]{Example}

\theoremstyle{remark}
\newtheorem{rem}[equation]{Remark}
\newtheorem{obs}[equation]{Observation}
\newcommand{\sprf}{\noindent{\it Proof.}}
\newcommand{\sqed}{\hfill\rule{1.3mm}{3mm}\medskip}

\newcounter{stareq}
\def\thestareq{\fnsymbol{stareq}}

\DeclareMathOperator{\BN}{\mathbb{N}} 
\DeclareMathOperator{\BZ}{\mathbb{Z}} 
\DeclareMathOperator{\BR}{\mathbb{R}} 

\newcommand{\bd}{\begin{description}}
\newcommand{\ed}{\end{description}}
\begin{document}

\bibliographystyle{amsalpha}

\title{Homotopical Complexity of 2D Billiard Orbits}

\author{Lee M. Goswick and N\'andor Sim\'anyi}
\address[Lee M. Goswick]{The University of Alabama at Birmingham\\
  Department of Mathematics\\
  1300 University Blvd., Suite 452\\
  Birmingham, AL 35294 U.S.A.}
\address[N\'andor Sim\'anyi]{The University of Alabama at Birmingham\\
  Department of Mathematics\\
  1300 University Blvd., Suite 452\\
  Birmingham, AL 35294 U.S.A.}
  
\email[Lee M. Goswick]{goswick@amadeus.math.uab.edu}
\email[N\'andor Sim\'anyi]{simanyi@math.uab.edu}

\thanks{The second author was partially supported by NSF grants DMS 0457168 
and DMS 0800538.}
 
\date{\today}

\begin{abstract}
Traditionally, rotation numbers for toroidal billiard flows are
defined as the limiting vectors of average displacements per time on
trajectory segments. Naturally, these creatures live in the
(commutative) vector space $\BR^n$, if the toroidal billiard is given
on the flat $n$-torus.  The billiard trajectories, being curves, often
getting very close to closed loops, quite naturally define elements of
the fundamental group of the billiard table. The simplest non-trivial
fundamental group obtained this way belongs to the classical Sinai
billiard, i.e. the billiard flow on the 2-torus with a single,
strictly convex obstacle (with smooth boundary) removed. This
fundamental group is known to be the group $\textbf{F}_2$ freely
generated by two elements, which is a heavily noncommutative,
hyperbolic group in Gromov's sense. We define the homotopical rotation
number and the homotopical rotation set for this model, and provide
lower and upper estimates for the latter one, along with checking the
validity of classically expected properties, like the density (in the
homotopical rotation set) of the homotopical rotation numbers of
periodic orbits.

The natural habitat for these objects is the infinite cone erected
upon the Cantor set $\text{Ends}(\textbf{F}_2)$ of all ``ends'' of the
hyperbolic group $\textbf{F}_2$. An element of
$\text{Ends}(\textbf{F}_2)$ describes the direction in (the Cayley
graph of) the group $\textbf{F}_2$ in which the considered trajectory
escapes to infinity, whereas the height function $t$ ($t \ge 0$) of
the cone gives us the average speed at which this escape takes place.

The main results of this paper claim that the orbits can only escape
to infinity at a speed not exceeding $\sqrt{2}$, and any direction
$e\in\text{Ends}(F_2)$ for the escape is feasible with any prescribed
speed $s$, $0\leq s\leq \sqrt{2}/2$.  This means that the radial upper
and lower bounds for the rotation set $R$ are actually pretty close to
each other.
\end{abstract}



\maketitle

\section{Introduction}
The concept of rotation number finds its origin in the study of the average
rotation around the circle $S^1$ per iteration, as classically defined by H.
Poincar\'e in the 1880's, when one iterates an orientation-preserving circle
homeomorphism $f:S^1 \rightarrow S^1$. This is equivalent to studying the
average displacement $(1/n)(F^n(x)-x)$ ($x \in \BR$) for the iterates
$F^n$ of a lifting $F:\BR \rightarrow \BR$ of $f$ on the universal
covering space $\BR$ of $S^1$. The study of fine homotopical properties of
geodesic lines on negatively curved, closed surfaces goes back at least to
Morse \cite{Mor24}. As far as we know, the first appearance of the concept of
homological rotation vectors (associated with flows on manifolds) was the
paper of Schwartzman \cite{Sch57}, see also Boyland \cite{Boy00} for further references
and a good survey of homotopical invariants associated with geodesic flows.
Following an analogous pattern, in \cite{BMS06} we defined the (still commutative)
rotation numbers of a $2D$ billiard flow on the billiard table $\mathbb{T}^2 =
\BR^2/\BZ^2$ with one convex obstacle (scatterer) $\mathcal{O}$
removed. Thus, the billiard table (configuration space) of the model in
\cite{BMS06} was $\mathcal{Q} = \mathbb{T}^2\setminus\mathcal{O}$.  Technically
speaking, we considered trajectory segments $\{x(t) | 0 \le t \le T\} \subset
\mathcal{Q}$ of the billiard flow, lifted them to the universal covering space
$\BR^2$ of $ \mathbb{T}^2$ (not of the configuration space $\mathcal{Q}$),
and then systematically studied the rotation vectors as limiting vectors of
the average displacement $(1/T)(\tilde{x}(T)-\tilde{x}(0)) \in \BR^2$
of the lifted orbit segments $\{\tilde{x}(t)|0 \le t \le T\}$ as $T
\rightarrow \infty$. These rotation vectors are still ``commutative'', for
they belong to the vector space $\BR^2$.

Despite all the advantages of the homological (or ``commutative'') rotation
vectors (i. e. that they belong to a real vector space, and this
provides us with useful tools to construct actual trajectories with prescribed
rotational behaviour), in our current view the ``right'' lifting of the
trajectory segments $\{x(t)|0 \le t \le T\} \subset \mathcal{Q}$ is to lift
these segments to the universal covering space of $\mathcal{Q} =
\mathbb{T}^2\setminus\mathcal{O}$, not of $\mathbb{T}^2$. This, in turn,
causes a profound difference in the nature of the arising rotation
``numbers'', primarily because the fundamental group $\pi_1(\mathcal{Q})$ of
the configuration space $\mathcal{Q}$ is the highly complex group
$\textbf{F}_2$ freely generated by two generators (see section 2 below or
\cite{Mas91}). After a bounded modification, trajectory segments $\{x(t)| 0 \le t
\le T\} \subset \mathcal{Q}$ give rise to closed loops $\gamma_T$ in
$\mathcal{Q}$, thus defining an element $g_T = [\gamma_T]$ in the fundamental
group $\pi_1(\mathcal{Q}) = \textbf{F}_2$. The limiting behavior of $g_T$ as
$T \rightarrow \infty$ will be investigated, quite naturally, from two
viewpoints:
\begin{enumerate}
   \item The direction ``$e$'' is to be determined, in which the element $g_T$
   escapes to infinity in the hyperbolic group $\textbf{F}_2$ or, equivalently,
   in its Cayley graph $\mathcal{G}$, see section 2 below. All possible
   directions $e$ form the horizon or the so called ideal boundary
   $\text{Ends}(\textbf{F}_2)$ of the group $\textbf{F}_2 =
   \pi_1(\mathcal{Q})$, see \cite{CoP93}.
   \item The average speed $s = \lim_{T \rightarrow \infty}
   (1/T)\text{dist}(g_T, 1)$ is to be determined, at which the element
   $g_T$ escapes to infinity, as $T \rightarrow \infty$. 
   These limits (or limits $\lim_{T_n \rightarrow \infty} 
   (1/T_n)\text{dist}(g_{T_n}, 1)$ for sequences of positive reals $T_n \nearrow \infty$) are nonnegative real numbers.
\end{enumerate}
The natural habitat for the two limit data $(s,e)$ is the infinite cone
\begin{displaymath}
   C = ([0, \infty) \times \text{Ends}(\textbf{F}_2))/(\{0\} \times \text{Ends}(\textbf{F}_2))
\end{displaymath}
erected upon the set $\text{Ends}(\textbf{F}_2)$, the latter supplied with the
usual Cantor space topology. Since the homotopical ``rotation numbers'' $(s,e)
\in C$ (and the corresponding homotopical rotation sets) are defined in terms
of the noncommutative fundamental group $\pi_1(\mathcal{Q}) = \textbf{F}_2$, these notions will be justifiably called homotopical or noncommutative rotation
numbers and sets.

In accordance with \cite{BMS06}, we will focus on systems with a so-called ``small
obstacle'', i.e., when the sole obstacle $\mathcal{O}$ is contained by some
circular disk of radius less than $\sqrt{2}/4$. Furthermore, again
following \cite{BMS06}, most of the time we will restrict our attention to the
so-called \textit{admissible orbits}, see the paragraph right after the proof
of Lemma \ref{bounded_dif} in \cite{BMS06}. The corresponding rotation set will be the so-called admissible homotopical rotation set $AR \subset C$. The homotopical rotation
set $R$ defined without the restriction of admissibility will be denoted by
$R$. Plainly, $AR \subset R$ and these sets are closed subsets of the cone
$C$.

The main results of this paper are theorems \ref{upper_bnd} and
\ref{lower_bnd}.  The former claims that the set $R$ is contained in
the closed ball $B(0,\,\sqrt{2})$ of radius $\sqrt{2}$ centered at the
vertex $0 = \{0\} \times \text{Ends}(\textbf{F}_2)$ of the cone
$C$. In particular, both sets $AR$ and $R$ are compact. The latter
result claims that the set $AR$ contains the closed ball
$B(0,\sqrt{2}/2)$ of $C$, provided that the radius $r_0$ of the sole
circular obstacle is less than $\sqrt{5}/10$. Thus, these two results
provide a pretty detailed description of the homotopical complexity of
billiard orbits: Any direction $e\in\text{Ends}(F_2)$ is feasible for
the trajectory to go to infinity, the speed of escape $s$ cannot be
bigger than $\sqrt{2}$, whereas any speed $s$, $0\leq s\leq
\sqrt{2}/2$, is achievable in any direction $e\in\text{Ends}(F_2)$.
Example \ref{sharp_exmp} shows that, in sharp contrast with the
expectations and the analogous results for the commutative rotation
numbers in \cite{BMS06}, the star-shaped set $R$ is not contained in
the unit ball $B(0,1)$ of $C$: it contains some radii of length
$\sqrt{2}$, thus the upper estimate of Theorem \ref{upper_bnd}, at
least as a direction independent upper bound for the radial size of
$R$, is actually sharp.

Finally, in the concluding Section \ref{conc_sec} we present a corollary (Theorem \ref{htop_thm}) of the proofs of Section \ref{main_sec} and make a few remarks. The theorem provides an effective constant as an upper bound for the topological entropy
$h_{top}(r_0)$ of the billiard flow, where $r_0$ is the radius of the sole
circular obstacle. The upper bound we obtain is explicit, unlike the one
obtained in \cite{BFK98} for the topological entropy of the flow.

Remark \ref{periodic_rem} asserts what is always expected for ``decent'' dynamical systems regarding the relation between homotopical rotation sets and periodic orbits:
the homotopical rotation numbers of periodic admissible orbits form a dense
subset in $AR$.

Finally, remarks \ref{n_obs_rem}--\ref{dim_rem} briefly outline the possibilities of
some interesting follow-up research, namely the investigation and understanding of
the homotopical rotation numbers for $2D$ toroidal billiards with $N$ round
obstacles.

\section{Main Results}\label{main_sec}
\subsection*{Lower Estimate for the Homotopical Rotation Set}

The configuration space $\mathcal{Q}$ (the billiard table) of our system is
the punctured $2D$-torus $\mathcal{Q} = \mathbb{T}^2\setminus \mathcal{O}$,
where the removed obstacle $\mathcal{O}$ is the open disk of radius $r_0$, $0
< r_0 < \sqrt{2}/4$, centered at the origin $(0, 0)$. (For simplicity
we assume that the obstacle is a round disk, though this is only an
unimportant technical condition, see Remark \ref{arb_obs_rem} below.)
The upper bound of $\sqrt{2}/4$ is exactly the condition of having a so-called ``small
obstacle'' in the sense of \cite{BMS06}.

The fundamental group $\pi_1(\mathcal{Q})$ of $\mathcal{Q}$ is
classically known to be the group $\textbf{F}_2 = \langle a, b
\rangle$, freely generated by the elements $a$ and $b$, see, for
example, \cite{Mas91}.  Perhaps the simplest way to see this is to consider a
simply connected fundamental domain
\begin{eqnarray}
  \mathcal{Q} &=& \{ x = (x_1, x_2) \in \BR^2 | 0 \le x_1, x_2 \le
  1, \ \text{dist}(x, \BZ^2) \ge r_0 \},
\end{eqnarray}
where the upper and lower horizontal sides of this domain are identified via
the equivalence relation $(x_1, 0) \sim (x_1, 1)$, $r_0 \le x_1 \le 1-r_0$,
and the left and right vertical sides are similarly identified via $(0, x_2)
\sim (1, x_2)$ for $r_0 \le x_2 \le 1-r_0$.
\begin{figure}[h]
   \centerline{
   \includegraphics[width=0.6\textwidth, height=0.3\textheight]{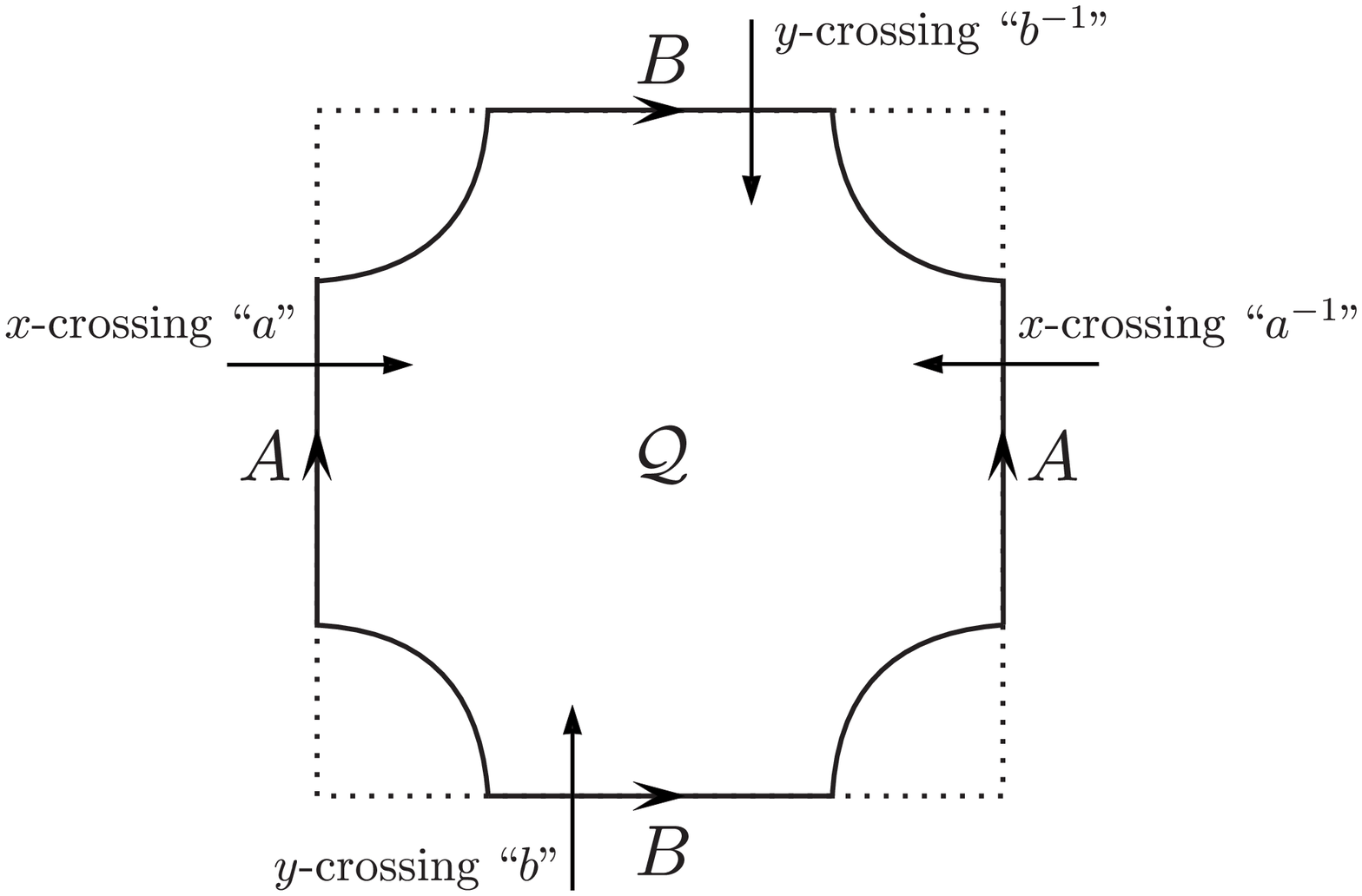}
   }
   \label{Figure 1}
   \caption{}
\end{figure}
The domain $\mathcal{Q}$ is obtained by identifying the opposite sides $A-A$
and $B-B$, just as the arrows indicate.  This space is homeomorphic to the
topological space that we obtain by gluing together two copies of a closed
strip $S^1 \times [-1,1]$ by identifying the rectangle $R_1 = [-1/10,
1/10] \times [-1,1]$ (in the first copy) with the same rectangle $R_2
= R_1$ (in the second copy) via the map $(x,y) \mapsto (y/10, 10x)$,
$|x| \le 1/10$, $|y| \le 1$, see Fig. 2.
\begin{figure}[h]
   \centerline{
   \includegraphics[width=0.6\textwidth, height=0.25\textheight]{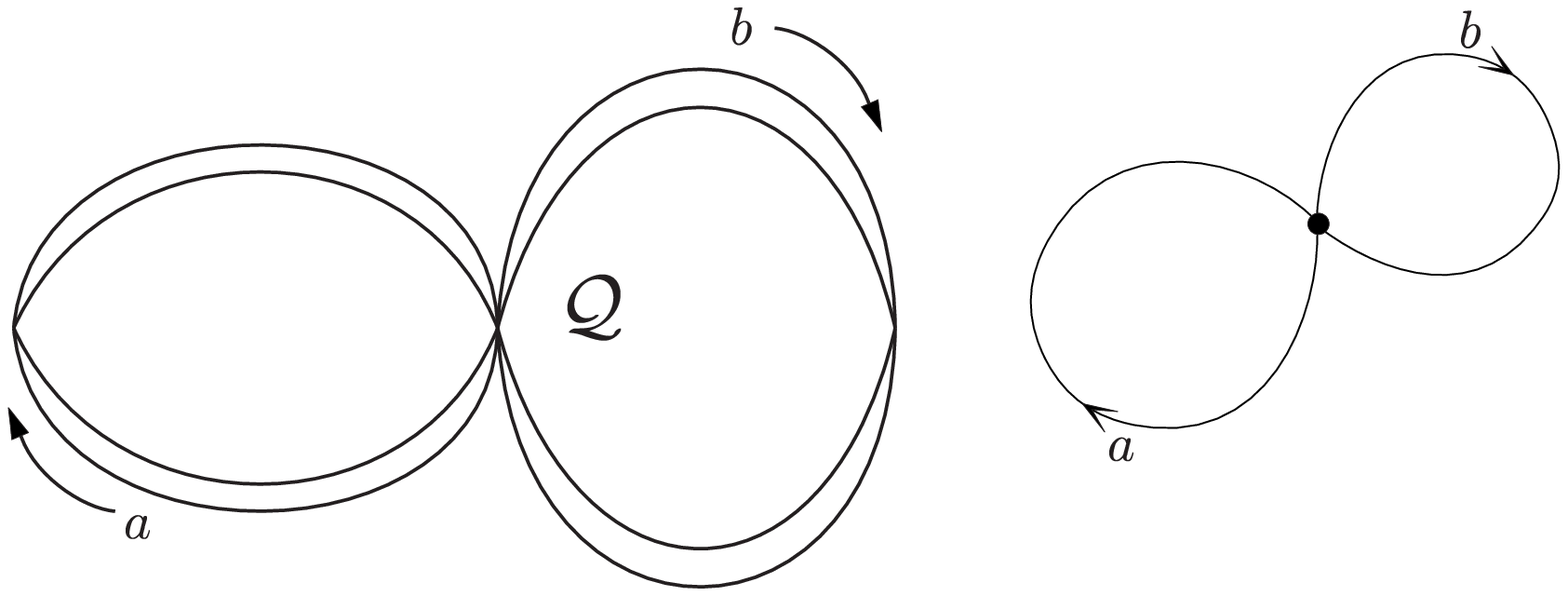}
   }
   \label{Figure 2}
   \caption{}
\end{figure}
The space $\mathcal{Q}$ is homotopically equivalent to the ``bouquet'' of two
circles, see the right part of Fig. 2. The fundamental group of the latter
space is classically known to be (see \cite{Mas91}) the group $\textbf{F}_2 =
\langle a, b \rangle$ freely generated by two elements ``$a$'' and ``$b$'', so
that ``$a$'' corresponds to making a loop along the first circle (in some
selected direction), whereas the generator ``$b$'' corresponds to making a
similar loop along the other circle. Clearly, these two generators correspond
to the so-called $x$- and $y$-crossings of curves (see Fig. 1). An
$x$-crossing ``$a$'' occurs when a smooth curve $\gamma(t) = (\gamma_1(t),
\gamma_2(t))$ intersects a line $\gamma_1(t) = k$ ($k \in \BZ$) with
$\dot{\gamma_1}(t) > 0$, while a $y$-crossing ``$b$'' takes place when
$\gamma(t) = (\gamma_1(t), \gamma_2(t))$ intersects a line $\gamma_2(t) = k$
($k \in \BZ$) with $\dot{\gamma_2}(t) > 0$. The ``inverse crossings''
$a^{-1}$ and $b^{-1}$ occur when the corresponding derivatives are negative.
We may assume that all these crossings are transversal. More precisely, we may
restrict our studies to such curves.

Our general goal is to study the large scale behavior of ``admissible''
billiard orbit segments $\pi(x(t)) = \pi((x_1(t), x_2(t)))$, $0 \le t \le T$,
$\text{dist}(x(0), (0,0)) = r_0$, as $T \rightarrow \infty$. Here,
``admissibility'' is understood in the sense of \cite{BMS06}, which means the
following: Denote by $\textbf{k}_0, \textbf{k}_1, \ldots, \textbf{k}_n \in
\BZ^2$ the centers of the obstacles $\mathcal{O}_{\textbf{k}_i}$ at whose
boundaries the lifted orbit segment $x(t)$, $0 \le t \le T$, is reflected,
listed in time order.  Admissibility means that the following three conditions
are satisfied:
\begin{enumerate}
   \item $\textbf{k}_0 = (0,0)$,
   \item for any $1 \le i \le n$, only the obstacles
   $\mathcal{O}_{\textbf{k}_{i-1}}$ and $\mathcal{O}_{\textbf{k}_i}$ intersect the
   convex hull of these two obstacles,
   \item for any $1 \le i \le n-1$, the obstacle $\mathcal{O}_{\textbf{k}_i}$ is
   disjoint from the convex hull of $\mathcal{O}_{\textbf{k}_{i-1}}$ and
   $\mathcal{O}_{\textbf{k}_{i+1}}$.
\end{enumerate}

A crucial result of \cite{BMS06}, Theorem 2.2 claims the existence of orbits
with any prescribed (finite or infinite) admissible itinerary
$\left(\textbf{k}_n\right)_{n=N_1}^{N_2}$.

In this paper we always consider the obstacles to be closed, i.e., containing
their boundaries. Whenever dealing with the so called admissible orbits,
we shall restrict ourselves to studying only
\renewcommand{\labelenumi}{(\Alph{enumi})}
\begin{enumerate}
   \item special admissible billiard orbit segments, the so called strongly admissible
   orbit segments, for which the above discrete itinerary
   \begin{displaymath}
      (\textbf{k}_0, \textbf{k}_1, \ldots, \textbf{k}_n)
   \end{displaymath}
   has the additional property that
   $\text{dist}(\textbf{k}_{i-1}, \textbf{k}_i) \leq \sqrt{2}$ for $i = 1, 2, \ldots, n$.
\end{enumerate}
We are primarily interested in discovering the asymptotic behavior of the
above segments $\{\pi(x(t))| 0 \le t \le T\}$ from the viewpoint of the
fundamental group $\pi_1(\mathcal{Q})$, as $T \rightarrow \infty$.  The first
question that arises here is how to measure the large-scale motion in
$\pi_1(\mathcal{Q})$ that is naturally associated with $\{\pi(x(t))| 0 \le t
\le T\}$? In order to answer this question, we first consider the so-called
Cayley graph $\mathcal{G} = (\mathcal{V}, \mathcal{E})$ of the group
$\pi_1(\mathcal{Q}) = \textbf{F}_2$ determined by the symmetric system of
generators $\mathcal{A} = \{a, a^{-1}, b, b^{-1}\}$. The vertex set
$\mathcal{V}$ of the Cayley graph $\mathcal{G}$ is, by definition, the
underlying set of the group $\textbf{F}_2$. We say that an oriented edge of type
$l\in \mathcal{A}$ goes from the element $w_1 \in \mathcal{V}$ to the element
$w_2 \in \mathcal{V}$ if $w_1l = w_2$. The arising oriented graph consists of
pairs of oppositely oriented edges $l, l^{-1}$. Other than the these cycles of
length 2, there are no cycles in the Cayley graph $\mathcal{G}$. If we
identify the opposite edges, then, obviously, we obtain a tree in which every
vertex has degree 4 (a so-called 4-regular tree). The graph $\mathcal{G}$ is
considered a rooted tree with root $1\in\mathcal{V}$. (The identity element
$1$ of the group $F_2$.)

On the set $\mathcal{V} = \textbf{F}_2$ a natural way to measure the distance
$d(x,y)$ between two vertices $x, y$ is to use the graph distance,
i.e., the length of the shortest path (the only simple path)
connecting $x, y$. Two facts are immediately clear about this
distance $d(\cdot, \cdot)$:
\renewcommand{\labelenumi}{(\arabic{enumi})}
\begin{enumerate}
   \item $d(1, w) = \|w\|$ is the so-called length of the word $w$,
   i.e., the overall number of letters $l \in \mathcal{A}$ that are
   needed to express $w$ in its shortest form,
   \item the metric $d(\cdot, \cdot)$ is left-invariant (for the
   whole Cayley graph $\mathcal{G}$ is invariant under the left regular
   action of $\textbf{F}_2$ on $\mathcal{V} = \textbf{F}_2$).
\end{enumerate}
Secondly, the correct way to define the direction in which a trajectory in
$\mathcal{V}$ goes to infinity is to use the so-called ``ends'' of the
hyperbolic group $\textbf{F}_2$ (see \cite{CoP93}). An end of $\textbf{F}_2$ is an
infinite, simple (not self-intersecting) path, i.e., an infinite branch $W =
(w_0, w_1, w_2, \ldots)$ where $w_i \in \textbf{F}_2$, $w_0 = 1$, $l_k =
w^{-1}_{k-1}w_k \in \mathcal{A}$, $k = 1, 2, \ldots$, $w_k \neq w_l$ for $k
\neq l$, or, equivalently, $l_k^{-1} \neq l_{k+1}$ for all $k \in \BN$. The
set of all ends $\text{Ends}(\textbf{F}_2)$ of $\textbf{F}_2$ will be denoted
by $E$. The elements $W$ of $E$ (as above) are uniquely determined by the
infinite sequence $(l_1, l_2, \ldots) \in \mathcal{A}^{\BN}$, where $l_k^{-1}
\neq l_{k+1}$ for all $k \in \BN$. In this way the set $E =
\text{Ends}(\textbf{F}_2)$ is identified with a closed subset of the product
space $\mathcal{A}^{\BN}$ and inherits from $\mathcal{A}^{\BN}$ its natural
product space (a Cantor set) topology. The set $E=\text{Ends}(\textbf{F}_2)$
with this topology is also called the \textit{horizon}, or the
\textit{ideal boundary} of the group $\textbf{F}_2$.

The large-scale behavior of the projected orbit segment
\[
\{\pi(x(t)) = \pi(x_1(t), x_2(t)) | 0 \le t \le T\} \subset \mathcal{Q}
\]
will be discovered by understanding
\renewcommand{\labelenumi}{(\alph{enumi})}
\begin{enumerate}
   \item in what direction $\pi(x(T))$ goes to $\infty$, when
   $\pi(x(T))$ is appropriately interpreted as an element of
   $\textbf{F}_2 = \pi_1(\mathcal{Q})$,
   \item at what speed $\pi(x(T))$ goes to infinity in
   $\textbf{F}_2$, i.e., how fast the distance $d(1, \pi(x(T)))$
   tends to infinity as a function of $T$.
\end{enumerate}
The natural phase space that incorporates the data of both (a) and
(b) is the cone
\begin{eqnarray}\label{cone_eq}
   C &=& ([0, \infty) \times E)/(\{0\} \times E)
\end{eqnarray}
erected upon the base $E$ that can be obtained from the product space $[0,
\infty) \times E$ by pinching together all points of the form $(0, e)$, $e \in
E$. The cone $C$ is clearly an open and dense subset of the compact metrizable
cone $\overline{C}$, in which the half open time interval $[0, \infty)$ is
replaced by the compact interval $[0, \infty]$. This means that the topology
of the cone $C$ can be induced by some complete separable metric (cf.  Theorem
4.3.23 in \cite{Eng89}), thus $C$ is a so-called Polish space. We will not use any
such actual metric inducing the topology of $C$, but will only measure the
distances of points from the vertex $0$ of $C$ by using the parameter function
$t$.

It is obvious that a subset $X$ of $C$ is compact if and only if
$X$ is closed and bounded, where boundedness of $X$ means the
boundedness of the distance function $t$ on $X$.

\subsection*{The Homotopical Rotation Set $R \subset C$ and the Admissible Homotopical
Rotation Set $AR \subset C$}
As we stated above, we shall study the asymptotic homotopical behavior of the
billiard trajectory segments $\pi(x(t)) = \pi(x_1(t), x_2(t))$, $0 \le t \le
T$, $x(0) \in \partial\mathcal{O}_{(0,0)}$, i.e., $d(x(0), (0,0)) = r_0$,
$x(T) \in \partial\mathcal{O}_{\textbf{k}_n}$, as $T \rightarrow \infty$.
Denote by $0 = t_0 < t_1 < t_2 < \cdots < t_n = T$ the times when
$d(x(t), \BZ^2) = r_0$, and let $x(t_i) \in \partial\mathcal{O}_{\textbf{k}_i}$,
$i = 0, 1, \ldots, n$, $\textbf{k}_0 = (0,0)$.
With this orbit segment $(x(0), \dot{x}(0), n)$ we naturally
associate an element $w = w(x(0),\dot{x}(0), n)\in\textbf{F}_2$
of the fundamental group $\pi_1(\mathbb{T}^2\setminus\mathcal{O})=\textbf{F}_2$
in the following way: We record the times $0 < \tau_1 < \tau_2 < \cdots < \tau_k < T$
when at least one of the two coordinates $x_1(\tau), x_2(\tau)$ is an
integer.
   \footnote[2]{It follows from the transversality condition
   (imposed on the piecewise smooth curve $x(t)$, $0 \le t \le T$)
   that the set of points to be listed above is discrete and closed,
   hence finite. Thus, the above finite listing $\{\tau_1, \tau_2,
   \ldots, \tau_k\}$ can indeed be done. This restriction only discards
   horizontal and vertical periodic trajectories with period $2$, bouncing back
   and forth between two neighboring obstacles at unit distance from each other.
   All these periodic orbits are trivial: they stay bounded in the group
   $\textbf{F}_2$.} 
If $x_1(\tau_i) \in \BZ$ and $\varepsilon_i =\text{sgn}\left[\left(d/d\tau\right)\,x_1(\tau)|_{\tau=\tau_i}\right]$,
then we take $w_i = a^{\varepsilon_i}$, while for $x_2(\tau_i) \in \BZ$ 
and $\varepsilon_i = \text{sgn}\left[\left(d/d\tau\right)\,x_2(\tau)|_{\tau=\tau_i}\right]$ 
we take $w_i = b^{\varepsilon_i}$. The first crossing will be called an $x$-crossing
$a^{\varepsilon_i}$, while the second crossing will be called a $y$-crossing $b^{\varepsilon_i}$, see also Fig. 1. The word $w = w(x(0), \dot{x}(0), T)$ is then
defined as the product $w = w_1w_2\ldots w_k$. We can now make the following 
observation:
\begin{obs}
   The billiard orbit segment $(x(0), \dot{x}(0), n) =
   \{\pi(x(t))|0 \le t \le T\}$ can be made a closed curve (a loop) in
   $\mathbb{T}^2\setminus\mathcal{O}$ by adding to it a bounded
   extension (beyond $T$). This bounded addition will only modify the
   word $w = w_1w_2\ldots w_k = w(x(0), \dot{x}(0), n)$ (defined
   above) by a bounded right multiplier, but all modifications have
   no effect on the asymptotic behavior of $w$ as $T \rightarrow
   \infty$, see Lemma \ref{bounded_dif} below.
\end{obs}

\renewcommand{\labelenumi}{(\arabic{enumi})}
\begin{df}\label{lim_pt_df}
   Let $x_i = \{x_i(t)|0 \le t \le T_i \}$ (i =1, 2, 3,
   \ldots) be an infinite sequence of piecewise smooth continuous curves in
   $\mathbb{T}^2\setminus\mathcal{O}$ with all transversal $x$- and
   $y$-crossings and $\lim_{i \rightarrow \infty} T_i = \infty$. We
   say that the point $(t,e) \in C$ of the cone $C$ is the limiting
   point of the sequence $(x_i)_{i = 1}^{\infty}$ if
   \begin{enumerate}
      \item $w(x_i) \rightarrow e$, as $i \rightarrow \infty$, and
      \item $\lim_{i \rightarrow \infty} (1/T_i)\|w(x_i)\| = t$.
   \end{enumerate}
\end{df}

\begin{lm} \label{bounded_dif}
   Let $x_i = \{x_i(t)|0 \le t \le T_i \}$ and $y_i = \{y_i(t)|0 \le t \le
   \tilde{T}_i \}$ be two infinite sequences of piecewise smooth continuous curves
   fulfilling the conditions of Definition \ref{lim_pt_df}, in particular,
   with $\lim_{i \rightarrow \infty}T_i=\lim_{i \rightarrow \infty}\tilde{T}_i=\infty$.
   Assume that $x_i$ and $y_i$ differ only by a bounded terminal segment, i.e.,
   there exists a bound $K > 0$ such that $|T_i -\tilde{T}_i| \le K$ and $x_i(t)
   \neq y_i(t)$ imply $T_i - t \le K$.
   Finally, assume that $(t,e) \in C$ is the limiting point of the sequence
   $(x_i)_{i=1}^{\infty}$. Then $(t,e)$ is also the limiting point of the
   sequence $(y_i)_{i=1}^{\infty}$.
\end{lm}
\begin{proof}
   Our boundedness hypothesis implies that there are
   words $w_i \in \textbf{F}_2$ and a constant $K_1$ such that
   \begin{eqnarray}
      w(y_i) = w(x_i)w_i \label{seq_mult} \\
      \|w_i\| \le K_1, \label{seq_bound}
   \end{eqnarray}
   for $i = 1, 2, 3, \ldots$. The assumed relation $w(x_i)
   \rightarrow e$ and \eqref{seq_mult}--\eqref{seq_bound}
   imply that $w(y_i) \rightarrow e$, as $i \rightarrow \infty$.
   Similarly, the sequences
   $(\|w(x_i)\| - \|w(y_i)\|)_{i = 1}^{\infty}$ and $(T_i -
   \tilde{T}_i)_{i=1}^{\infty}$ are bounded, hence the relation
\[
\lim_{i \rightarrow \infty} (1/T_i)\|w(x_i)\| = t
\]
implies
\[
\lim_{i \rightarrow \infty} (1/\tilde{T}_i)\|w(y_i)\| = t.
\]
\end{proof}

\begin{df}\label{admis_def}
   The homotopical rotation set $R \subset C$ is defined as all
   possible limiting points of sequences of orbit segments $x_i =
   \{x_i(t)| 0 \le t \le T_i\}$ with $T_i \rightarrow \infty$.
   Similarly, the admissible homotopical rotation set $AR \subset C$ is the set
   of all possible limiting points of sequences of admissible
   billiard orbit segments. It is clear that $AR \subset R$ and both
   are closed subsets of the cone $C$.
\end{df}

\begin{df}
   For a given forward orbit $x = \{x(t)| t \ge 0\}$ the homotopical rotation set
   $R(x)$ of $x$ is defined as the set of all possible limiting points $(t,e) \in
   C$ of sequences of orbit segments $x_i = \{x(t)| 0 \le t \le T_i \}$ (these
   are initial segments of $x$) with $\lim_{i \rightarrow \infty} T_i = \infty$.
   Plainly, $R(x)$ is a closed subset of the cone $C$.  Theorem \ref{upper_bnd}
   below will ensure that $R(x)$ is a non-empty, compact set. In the case $|R(x)| = 1$,
   i.e., when $R(x)$ is a singleton, the sole element of $R(x)$ will be called
   the homotopical rotation number of the forward orbit $x$.
\end{df}

\begin{rem}
   For the definition of admissible billiard orbits, please see the above
   definition in this section or the definition of admissibility immediately
   preceding Theorem 2.2 in \cite{BMS06}. Also, please compare the definition of $R$
   and $AR$ here with the analogous definitions at the beginning of section 3 of
   \cite{BMS06}.
\end{rem}

\begin{rem}
   We also note that any symbolic admissible itinerary
   $(k_0, k_1, \ldots)$ (finite or infinite) can actually be realized
   by a genuine billiard orbit. Please see Theorem 2.2 in \cite{BMS06}.
\end{rem}

The first result of this paper is a uniform upper bound for the radial size of
the full homotopical rotation set $R$.
\begin{thm}\label{upper_bnd}
   The homotopical rotation set $R$ is contained in the closed ball
   $B(0,\sqrt{2})$ centered at the vertex $0$ of the cone $C$ with radius
   $\sqrt{2}$. In particular, the set $R$ is compact.
\end{thm}
\begin{proof}
   Throughout this proof we will be dealing exclusively with orbit
   segments $x(t)=\left(x_1(t),x_2(t)\right)$ ($0\le t\le T$) lifted to the
   covering space
   \begin{displaymath}
      \tilde{\mathcal{Q}}
      =\left\{x\in\BR^2\big|\; \rm{dist}(x,\mathbb{Z}^2)\ge r_0\right\}
   \end{displaymath}
   of the configuration space $\mathcal{Q}$. The trivial, periodic orbits
   bouncing back and forth horizontally (vertically) between two neighboring
   obstacles (i. e. two obstacles with their centers at unit distance from each
   other) will be excluded from our considerations.

   First of all, we make a simple observation:
   \begin{lm}
      Let $\tau_1$ and $\tau_2$ ($0\le\tau_1<\tau_2\le T$)
      be the time moments of two consecutive $x$-crossings of the orbit segment
      $x(t)=\left(x_1(t),x_2(t)\right)$ ($0\le t\le T$). We claim that
      \begin{displaymath}
         \int_{\tau_1}^{\tau_2}\left|\dot{x}_1(t)\right|dt \ge 1.
      \end{displaymath}
   \end{lm}
   \begin{proof}
      Without loss of generality we may assume that
      $\dot{x}_1(\tau_1)>0$. Let $x_1(\tau_1)=k\in\BZ$. Then
      $x_1(\tau_2)=k+1$ or $x_1(\tau_2)=k$. In the former case we are done, so we
      assume that $x_1(\tau_1)=k=x_1(\tau_2)$. Clearly, in this case
      $\dot{x}_1(\tau_2)<0$. In order for the particle to change its positive
      horizontal momentum $\dot{x}_1(\tau_1)$ to the negative value of
      $\dot{x}_1(\tau_2)$, it is necessary for the particle to cross the median
      $x_1=k+1/2$ of the vertical strip $k\le x_1\le k+1$, for any collision
      on the left side of this strip can only increase the horizontal momentum. This
      observation yields the claimed lower estimate.
   \end{proof}

   \begin{rem}
      The counterpart of the lemma providing a similar lower estimate
      \begin{displaymath}
         \int_{\tau_1}^{\tau_2}\left|\dot{x}_2(t)\right|dt \ge 1
      \end{displaymath}
      between two consecutive $y$-crossings is also true, obviously.
   \end{rem}

   Denote by $N$ the overall number of $x$- and $y$-crossings (counted without
   the sign) on the considered orbit segment
   $\left\{x(t)=\left(x_1(t),x_2(t)\right)\big|\; 0\le t\le T\right\}$. The above
   lemma gives us the upper estimate
   \begin{displaymath}
      N \le \int_0^T
      \left(\left|\dot{x}_1(t)\right|+\left|\dot{x}_2(t)\right|\right)dt+2
   \end{displaymath}
   for the number $N$.
   Since $\left|\dot{x}_1(t)\right|+\left|\dot{x}_2(t)\right|\le\sqrt{2}$,
   we get that $N\le\sqrt{2}T+2$, that is, $N/T\le\sqrt{2}+2/T$, and this
   proves the theorem.
\end{proof}

\begin{exmp}\label{sharp_exmp}
   The upper bound $\sqrt{2}$ for the radial size of $R$ cannot be improved
   uniformly for all directions $e\in\rm{Ends}(\textbf{F}_2)$, as the following example
   shows: The ``smallness'' condition $r_0< \sqrt{2}/4$ precisely means that
   the corridor (strip)
   \begin{displaymath}
      S_0=\left\{x=(x_1,x_2)\in\BR^2\big|\; \sqrt{2}r_0\le x_2-x_1\le
      1-\sqrt{2}r_0\right\}
   \end{displaymath}
   is free of obstacles in the covering space
   \begin{displaymath}
      \tilde{\mathcal{Q}}
      =\left\{x\in\BR^2\big|\; \rm{dist}(x,\BZ^2)\ge r_0\right\}.
   \end{displaymath}
   In this corridor $S_0$, for any natural number $n$ we construct the periodic
   orbit (periodic after projecting it into $\mathcal{Q}$)
   \begin{displaymath}
      \left\{x^{(n)}(t)=\left(x^{(n)}_1(t),\, x^{(n)}_2(t)\right)\big|\; t\in\BR
      \right\}
   \end{displaymath}
   that has consecutive reflections at the points
   \begin{displaymath}
      \left(\dots,\,P_{-1},\,Q_{-1},\,P_0,\,Q_0,\,P_1,\,Q_1,\dots\right)
   \end{displaymath}
   (written in time order), where
   \begin{displaymath}
      \aligned
      P_k&=v_0+k(2n+1,\,2n+1), \\
      Q_k&=-v_0+(n,\,n+1)+k(2n+1,\,2n+1)
      \endaligned
   \end{displaymath}
   ($k\in\mathbb{Z}$) with $v_0=\left(-r_0/\sqrt{2},\,r_0/\sqrt{2}\right)$.
   The period length $T_n$ of of $x^{(n)}$ is
   \begin{displaymath}
      T_n=2\left||(n,\,n+1)-2v_0\right||=
      2\left(2n^2+2n+1+4r_0^2-2\sqrt{2}r_0\right)^{1/2}=2\sqrt{2}n+O(1),
   \end{displaymath}
   whereas this periodic orbit makes exactly $2n+1$ $x$-crossings $a$ and
   $2n+1$ $y$-crossings $b$ during one period. Thus, the word length
   \begin{displaymath}
      \left||w\left(\left\{x^{(n)}(t)\big|\; 0\le t\le T_n\right\}\right)\right||
   \end{displaymath}
   is equal to $4n+2$, therefore
   \begin{displaymath}
      \frac{\left||w\left(\left\{x^{(n)}(t)\big|\;
      0\le t\le T_n\right\}\right)\right||}
      {T_n}=\frac{4n+2}{2\sqrt{2}n+O(1)},
   \end{displaymath}
   and this quantity tends to $\sqrt{2}$, as $n\to\infty$.
\end{exmp}

The main result of this paper is an effective lower bound for the
set $AR$ and, consequently, for the full rotation set $R$:

\begin{thm}\label{lower_bnd}
   Assume that the radius $r_0$ of the sole obstacle is less than
   $\sqrt{5}/10$. We claim that the admissible rotation set $AR$ contains the
   closed ball $B(0,\,\sqrt{2}/2)\subset C$ of radius $\sqrt{2}/2$ centered at
   the vertex $0$ of the cone $C$.
\end{thm}
\begin{proof}
   The proof of this lemma will be subdivided into a few lemmas and
   observations. First of all, we observe

   \begin{obs}
      The imposed condition $r_0<\sqrt{5}/10$ is
      equivalent to requiring that the circular scatterer $\mathcal{O}_{(0,0)}$
      does not intersect the convex hull of the scatterers $\mathcal{O}_{(-1,-1)}$
      and $\mathcal{O}_{(0,1)}$. Therefore, under our condition of $r_0<\sqrt{5}/10$
      the following statements hold true:
      \renewcommand{\labelenumi}{(\roman{enumi})}
      \begin{enumerate}
         \item Every integer vector $\mathbf{k}\in\mathbb{Z}^2$ of length $1$
         or $\sqrt{2}$ is a vertex of the admissibility graph $G$ (please see the first paragraph after the proof of Lemma 2.5 in \cite{BMS06}), 
         i. e. the passage $\mathbf{k}$ is admissible;
         \item If $\mathbf{k}$ and $\mathbf{l}$ are two distinct integer vectors
         with norms $1$ or $\sqrt{2}$, then there is an oriented edge $\mathbf{k}\to\mathbf{l}$ in the admissibility graph $G$, that is, 
         in an admissible itinerary a passage $\mathbf{l}$ is permitted to follow a passage $\mathbf{k}$.
      \end{enumerate}
      \renewcommand{\labelenumi}{(\arabic{enumi})}
   \end{obs}

   The above statements are easily checked by an elementary inspection.

   At the core of the proof (of the theorem) is

   \begin{lm}\label{cons_obs_lm}
      Suppose that $n\ge 0$ is an integer and
      $\mathbf{k}_0,\mathbf{k}_1,\dots,\mathbf{k}_{n+2}$ ($\in\mathbb{Z}^2$)
      are centers of obstacles that are consecutively visited by the
      segment $S^{[0,T]}x$ of a strongly admissible orbit
      $S^{(-\infty,\infty)}x$, so that they are having the following
      properties:
      \begin{enumerate}
         \item The passage vectors $\mathbf{l}_i=\mathbf{k}_{i+1}-\mathbf{k}_i$ are
         equal to $(1,\,(-1)^{i+1})$ for $i=1,\dots,n$;
         \item The ``initial connector'' passage vector
         $\mathbf{l}_0=\mathbf{k}_1-\mathbf{k}_0$ is either $(1,0)$, or $(0,-1)$;
         \item The ``terminal connector''
         $\mathbf{l}_{n+1}=\mathbf{k}_{n+2}-\mathbf{k}_{n+1}$ is either $(1,0)$, or $(0,(-1)^n)$;
         \item If $\mathbf{l}_0=(0,-1)$, then the passage vector $\mathbf{l}_{-1}$
         (directly preceding $\mathbf{l}_0$ in the itinerary of $x$) has positive first coordinate and, if $\mathbf{l}_{n+1}=(0,(-1)^n)$, 
         then the passage vector $\mathbf{l}_{n+2}$ has positive first coordinate;
         \item $S^0x=x_0=x$ corresponds to the collision at
         $\mathcal{O}_{\mathbf{k}_0}$, while $x_T=S^Tx$ corresponds to the collision at $\mathcal{O}_{\mathbf{k}_{n+2}}$.
      \end{enumerate}
      (Note that, by admissibility, $\mathbf{l}_0=\mathbf{l}_{n+1}=(1,0)$ is not
      permitted in the case $n=0$.)

      We claim that the orbit segment $S^{[0,T]}x$ makes $n+1$ $x$-crossings ``$a$''
      and no $y$-crossings at all, with the only (possible) exception that the initial connector
      $\mathbf{l}_0=(1,0)$ (if it is $(1,0)$) may make a $y$-crossing $b$, just as the terminal connector
      $\mathbf{l}_{n+1}=(1,0)$ may make a $y$-crossing $b$ (when $n$ is odd), or 
      $\mathbf{l}_{n+1}=(1,0)$ may make a $y$-crossing $b^{-1}$ (when $n$ is even).
   \end{lm}

   \begin{proof} The lemma is proved by an elementary inspection, see also
      Figure 3 below.
   \end{proof}
   \begin{figure}[h]
      \centerline{
      \includegraphics[width=1.0\textwidth, height=0.3\textheight]{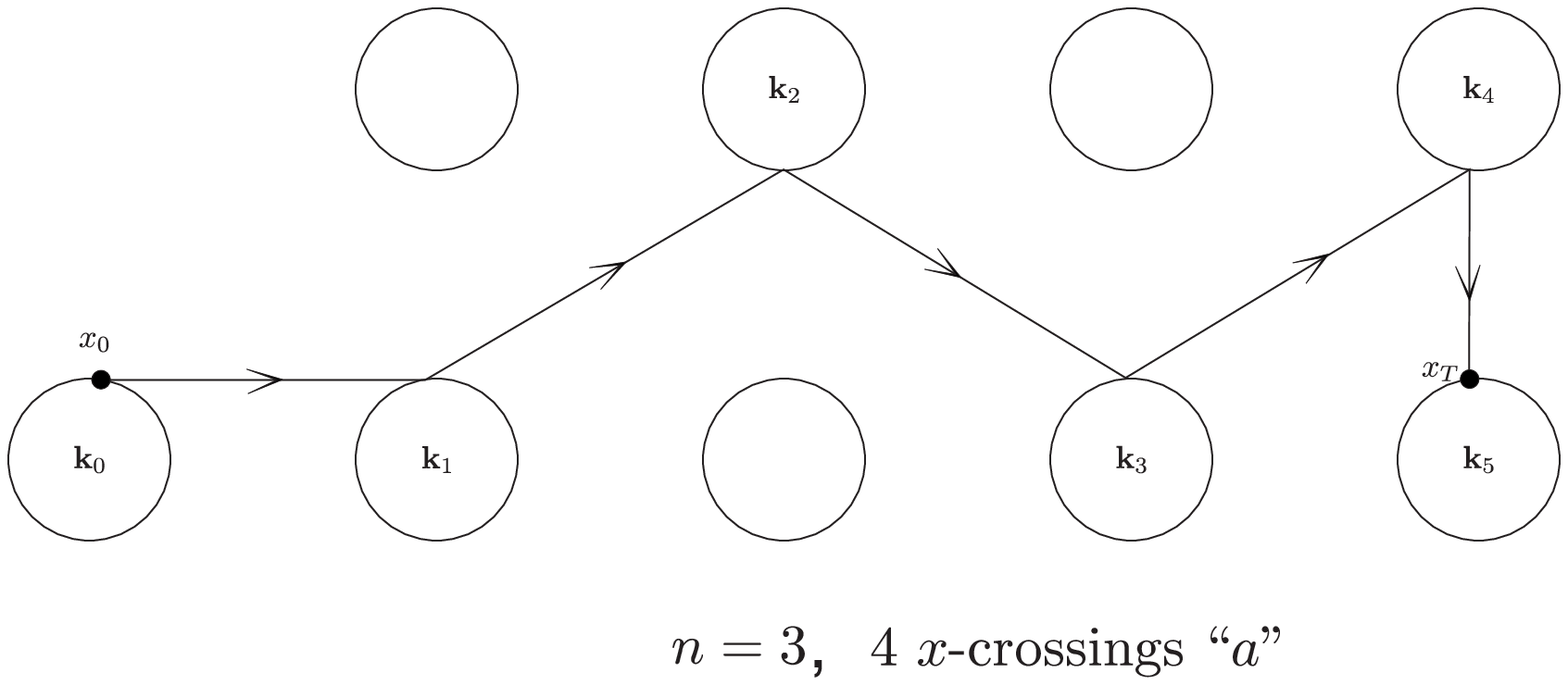}}
      \label{Figure 3}
      \caption{}
   \end{figure}

   \begin{rem}
      An admissible orbit segment $S^{[0,T]}x$ (described in the lemma above) will
      be called an ``$a^{n+1}$-passage'' with the connectors $\mathbf{l}_0$ and
      $\mathbf{l}_{n+1}$, where the first connector is called the ``initial
      connector'', while the latter one is called the ``terminal connector''. Observe
      that in this $a^{n+1}$-passage $S^{[0,T]}x$ the $x$-crossings ``$a$'' are in a
      natural, one-to-one correspondence with the reflections at the
      boundaries of $\mathcal{O}_{\mathbf{k}_1},\dots,\mathcal{O}_{\mathbf{k}_{n+1}}$, respectively.
      These reflections will be called the ``eigenreflections'' of the
      $a^{n+1}$-passage $S^{[0,T]}x$. The two reflections at the boundaries of
      $\mathcal{O}_{\mathbf{k}_0}$ and $\mathcal{O}_{\mathbf{k}_{n+2}}$ will not be
      considered as eigenreflections of this $a^{n+1}$-passage: The first one of
      them will actually be the last eigenreflection of a $b^m$-passage ($m\ne 0$)
      directly preceding the considered $a^{n+1}$-passage, while the second one
      will be the first eigenreflection of the $b^p$-passage ($p\ne 0$) directly
      following the $a^{n+1}$-passage $S^{[0,T]}x$. The shared passage vectors
      $\mathbf{l}_0$ and $\mathbf{l}_{n+1}$ will serve as connectors between the
      neighboring $a$- and $b$-passages. All passage vectors $\mathbf{l}_i$, used in
      this construction, have length $1$ or $\sqrt{2}$. We will say that the
      sequence of passage vectors
      $\sigma=(\mathbf{l}_0,\mathbf{l}_1,\dots,\mathbf{l}_{n+1})$ is the \textit{symbolic
      code} of the considered $a^{n+1}$-passage $S^{[0,T]}x$.
   \end{rem}

   \begin{rem}
      Clearly, similar statements are true on $a^m$-passages ($m<0$) and
      $b^m$-passages ($m\ne 0$). Also, in its current form of the lemma on
      $a^{n+1}$-passages, the first non-connector passage vector $\mathbf{l}_1=(1,1)$
      could have been $(1,-1)$, by appropriately reflecting all other passage
      vectors about the $x$-axis.
      \end{rem}

\begin{rem}
A few words are due here about the possible ``exceptional'' $b$ or $b^{-1}$ crossings of the initial
and/or terminal connectors, mentioned at the end of the claim of the lemma: If the initial
connector $\mathbf{l}_0=(1,0)$ happens to make an ``exceptional'' $y$-crossing $b$, then this crossing
will be counted as the last $y$-crossing of the $b^m$-passage ($m>0$) preceding the considered
$a^{n+1}$-passage. Similar statement can be said (mutatis mutandis) about a possible ``exceptional''
$y$-crossing ($b$ or $b^{-1}$) of the terminal connector $\mathbf{l}_{n+1}=(1,0)$.
\end{rem}

\begin{rem}
If $n>0$, then there are exactly $8$ different combinatorial possibilities for the symbolic code
$\sigma=(\mathbf{l}_0,\mathbf{l}_1,\dots,\mathbf{l}_{n+1})$ of an $a^{n+1}$-passage: The $x$ coordinates
of the connectors $\mathbf{l}_0$ and $\mathbf{l}_{n+1}$ can be $0$ or $1$ independently, whereas $\mathbf{l}_1$
can be $(1,1)$ or $(1,-1)$, also independently chosen from $\mathbf{l}_0$ and $\mathbf{l}_{n+1}$. However, for
$n=0$ there are only $6$ possibilities for $\sigma=(\mathbf{l}_0,\mathbf{l}_1)$:
\begin{enumerate}
\item $\mathbf{l}_0=(0,1)$, $\mathbf{l}_1=(0,-1)$;
\item $\mathbf{l}_0=(0,-1)$, $\mathbf{l}_1=(0,1)$;
\item $\mathbf{l}_0=(1,0)$, $\mathbf{l}_1=(0,1)$;
\item $\mathbf{l}_0=(1,0)$, $\mathbf{l}_1=(0,-1)$;
\item $\mathbf{l}_0=(0,1)$, $\mathbf{l}_1=(1,0)$;
\item $\mathbf{l}_0=(0,-1)$, $\mathbf{l}_1=(1,0)$.
\end{enumerate}
\end{rem}

   Consider an arbitrary element $w_\infty=\prod_{i=1}^\infty a^{n_i}b^{m_i}$
   (an infinite word) of the set $\rm{Ends}(F_2)$. For any natural number $N$ we
   want to construct a finite, admissible orbit segment $S^{[0,T_N]}x_N=x^{(N)}$,
   the associated word $w(x^{(N)})$ of which is
   $\prod_{i=1}^N a^{n_i}b^{m_i}:=w_N$, such that
   \begin{displaymath}
      \limsup_{N\to\infty}\frac{||w_N||}{T_N}\ge\frac{\sqrt{2}}{2}.
   \end{displaymath}
   By symmetry, we may assume that the considered word $w_\infty$ begins with a
   power of ``$a$'' (as the notations above indicate), and that $n_1>0$. We shall
   use Lemma \ref{cons_obs_lm} by successively concatenating the $a^{n_i}$- and
   $b^{m_i}$-passages ($i=1,\dots,N$) to obtain the admissible orbit segment
   $x^{(N)}=S^{[0,T_N]}x_N$ with the associated word
   \begin{displaymath}
      w(x^{(N)})=w_N=\prod_{i=1}^N a^{n_i}b^{m_i}.
   \end{displaymath}
   This will be achieved by constructing first the symbolic, admissible itinerary
   of $x^{(N)}$ containing only passage vectors
   $\mathbf{l}_j=\mathbf{k}_{j+1}-\mathbf{k}_j\in\mathbb{Z}^2$ of length $1$ and
   $\sqrt{2}$.

   For simplicity (and by symmetry) we assume that $n_1>0$.  First we
   construct the symbolic itinerary
   $(\mathbf{l}_0,\dots,\mathbf{l}_{n_1})$ of an $a^{n_1}$-passage by
   taking $\mathbf{l}_0=(1,0)$, $\mathbf{l}_j=(1,\,(-1)^{j+1})$ for
   $j=1,2,\dots,n_1-1$. The terminal connector $\mathbf{l}_{n_1}$ will
   be carefully chosen, depending on the parity of $n_1$ and the sign
   of the integer $m_1$. By symmetry we may assume that
   $\mathbf{l}_{n_1-1}=(1,\,-1)$, i. e. that $n_1$ is an odd number. In
   the construction of the terminal connector $\mathbf{l}_{n_1}$ and the
   symbolic itinerary
   $(\mathbf{l}_{n_1},\mathbf{l}_{n_1+1},\dots,\mathbf{l}_{n_1+|m_1|})$ of
   the subsequent $b^{m_1}$-passage we will distinguish between two,
   essentially different cases.

   \subsection*{Case I. $m_1>0$} In this case we take $\mathbf{l}_{n_1}=(0,1)$ and,
   furthermore, $\mathbf{l}_j=((-1)^{j},\,1)$ for
   $j=n_1+1,n_1+2,\dots,n_1+|m_1|-1$. The terminal connector $\mathbf{l}_{n_1+|m_1|}$
   of this $b^{m_1}$-passage will be carefully chosen by a coupling process
   (similar to the one that we are just describing here) to couple the
   $b^{m_1}$-passage with the subsequent $a^{n_2}$-passage.

   \subsection*{Case II. $m_1<0$} In this case we take $\mathbf{l}_{n_1}=(1,0)$,
   $\mathbf{l}_j=((-1)^{j+1},\,-1)$ for $j=n_1+1,n_1+2,\dots,n_1+|m_1|-1$. Again, the
   terminal connector $\mathbf{l}_{n_1+|m_1|}$ of this $b^{m_1}$-passage will be
   carefully chosen by a coupling process to couple the $b^{m_1}$-passage with
   the subsequent $a^{n_2}$-passage.

   It is clear that the above process can be continued (by changing whatever
   needs to be changed, according to the apparent mirror symmetries of the
   system) to couple together the subsequent $a^{n_1}$-, $b^{m_1}$-,
   $a^{n_2}$-, $b^{m_2}$-,$\dots$, $a^{n_N}$-, and $b^{m_N}$-passages. In this
   way we obtain the admissible symbolic itinerary
   $(\mathbf{l}_0,\mathbf{l}_1,\dots,\mathbf{l}_{||w_N||})$ of a potential admissible
   orbit segment $x^{(N)}$ with the associated word
   \begin{displaymath}
      w(x^{(N)})=w_N=\prod_{i=1}^N a^{n_i}b^{m_i}.
   \end{displaymath}
   The existence of such an admissible orbit segment $x^{(N)}$ is guaranteed by
   Theorem 2.2 of \cite{BMS06}, using an orbit length minimizing principle in the construction. This
   means that a required orbit segment $x^{(N)}$ can be obtained by minimizing
   the length of all piecewise linear curves (broken lines)
   $P_0P_1\dots P_{L+1}$ ($L=||w_N||=\sum_{i=1}^N(|n_i|+|m_i|)$) for which the
   corner points $P_j$ belong to the obstacle (the closed disk)
   $\mathcal{O}_{\mathbf{k}_j}$ ($j=0,1,\dots,L+1$) with
   $\mathbf{k}_j=\sum_{i=0}^{j-1}\mathbf{l}_i$. Clearly, the length $T_N$ of the
   arising orbit segment $x^{(N)}$ is less than the length of the broken line
   connecting the consecutive centers $\mathbf{k}_j$ ($0\le j\le L+1$) of the
   affected obstacles, and this latter number is
   $\sum_{j=0}^L||\mathbf{l}_j||\le\sqrt{2}(L+1)$. Thus, we get that
   \begin{displaymath}
      \frac{||w_N||}{T_N}>\frac{L}{\sqrt{2}(L+1)},
   \end{displaymath}
   and this proves Theorem \ref{lower_bnd}.

   We note that if a point $(t,e)$ turns out to be a limiting point of a sequence
   of admissible orbit segments with passage vectors of length $1$ or $\sqrt{2}$,
   then any other point $(t_1, e) \in C$ with $0 \le t_1 \le t$ is also such a
   limiting point. Indeed, by inserting the necessary amount of ``idle
   sequences'' $\alpha\alpha^{-1}\alpha\alpha^{-1}\cdots$ in the itinerary, we can decrease the
   ratios $\|w(\sigma)\|/T(\sigma)$ (and their limits) as we wish. This
   finishes the proof of the theorem.
\end{proof}

An immediate consequence of the last argument is

\begin{cor}
   The set $AR$ is star-shaped from the view point $(0,0)\in C$, i.e.,
   $(t,e) \in AR$ and $0 \le t_1 \le t$ imply that $(t_1, e) \in AR$.
\end{cor}

The concluding result of this section shows that the lower estimate
$\sqrt{2}/2$ for the radial size of $AR$ is actually sharp, at least in some
directions $w_\infty\in\rm{Ends}(\textbf{F}_2)$ and in the \textit{small obstacle limit}
$r_0\to 0$.

\begin{prp}
   Consider the direction
   \begin{displaymath}
      w_\infty=aba^{-1}b^{-1}aba^{-1}b^{-1}\dots \in\rm{Ends}(\textbf{F}_2),
   \end{displaymath}
   i. e. the ``infinite power'' of the commutator element
   $[a,b]=aba^{-1}b^{-1}$. We claim that the radial size
   \begin{displaymath}
      \sigma=\sigma(r_0,w_\infty)=\sup\left\{t\in\mathbb{R}_+\big|\;
      (t,\,w_\infty)\in R\right\}
   \end{displaymath}
   of the full rotation set $R$ in the direction of $w_\infty$ has the limiting
   value $\sqrt{2}/2$, as $r_0\to 0$. In particular, similar statement holds true
   for the radial size of the smaller, admissible rotation set $AR$ in the same
   direction. Thus, the lower estimate $\sqrt{2}/2$ for the radial size of $AR$
   (of $R$) in this direction cannot be improved in the small obstacle limit
   $r_0\to 0$. We recall that, according to Theorem \ref{lower_bnd} above,
   $\sigma(r_0, e)\ge \sqrt{2}/2$ for all $r_0$ and all
   $e\in\rm{Ends}(\textbf{F}_2)$.
\end{prp}
\begin{proof}
   Consider an infinite sequence $(x_n)_{n=1}^\infty$ of orbit segments
   \begin{displaymath}
      x_n=\left\{x_n(t)\big|\; 0\le t\le T_n\right\}
   \end{displaymath}
   with $T_n\nearrow\infty$, $w(x_n)=(aba^{-1}b^{-1})^{k_n}$,
   $k_n\nearrow\infty$, and
   \begin{displaymath}
      \lim_{n\to\infty}\frac{4k_n}{T_n}=\sigma(r_0,\, w_\infty).
   \end{displaymath}
   We may assume that the relevant $x$-crossings ``$a$'' of $x_n$ (relevant in the
   sense that their symbol $a$ remains in the associated word $w(x_n)$ after all
   possible shortenings) take place between the obstacles at $(0,0)$ and
   $(0,-1)$, the relevant $y$-crossings ``$b$'' occur between the obstacles at
   $(0,0)$ and $(1,0)$, the relevant $x$-crossings ``$a^{-1}$'' happen between the
   obstacles at $(0,0)$ and $(0,1)$ and, finally, the relevant $y$-crossings
   ``$b^{-1}$'' take place between the obstacles at $(0,0)$ and $(-1,0)$, i. e.
   $x_n$ circles around the central obstacle $\mathcal{O}_{(0,0)}$
   counterclockwise. The proof of the inequality
   \begin{displaymath}
      \limsup_{r_0\to 0} \sigma(r_0,\, w_\infty)\le\frac{\sqrt{2}}{2}
   \end{displaymath}
   will be based on the following, elementary geometric observation:

   \begin{lm}
      Let $N$ be a natural number and
      \begin{displaymath}
         \Gamma=\left\{\gamma(t)\big|\; 0\le t\le T\right\}
      \end{displaymath}
      be a piecewise linear curve (a broken line) in $\BR^2$ parametrized
      with the arc length, enjoying the following properties:
      \begin{enumerate}
         \item every vertex (corner) of $\Gamma$ is an integer point;
         \item $(0,0)\not\in\Gamma$;
         \item $\Gamma$ winds around the origin at least $N$ times, i. e.
         \begin{displaymath}
            \int_0^T \dot{\omega}(t)dt\ge 2\pi N,
         \end{displaymath}
         where $\omega(t)$ is the angular polar coordinate of $\gamma(t)$.
      \end{enumerate}
      We claim that $T\ge 4\sqrt{2}N$, and the equation holds if and only if
      $\Gamma$ connects the lattice points $(1,0)$, $(0,1)$, $(-1,0)$, and $(0,-1)$
      in this cyclic order.
   \end{lm}

   Since the proof of this result is a simple, elementary geometric argument
   (though with a little bit tedious investigation of a few cases), we omit it,
   and immediately turn to the proof of the proposition.

   For $n=1,2,\dots$ we define a broken line $\Gamma_n$, fulfilling all
   conditions of the previous lemma with $N=k_n$, by
   \renewcommand{\labelenumi}{(\alph{enumi})}
   \begin{enumerate}
      \item considering all centers $c_1,c_2,\dots,c_m$ ($c_i\in\mathbb{Z}^2$)
      of the obstacles visited by $x_n$ in the time order $x_n$ visits them;
      \item dropping the possible appearances of the origin from the above sequence
      $c_1,c_2,\dots,c_m$;
      \item constructing $\Gamma_n$ by connecting the lattice points
      $c_1,c_2,\dots,c_m$ (in this order) and, by adding a bounded extension to
      $\Gamma_n$ if necessary, ensuring that $\Gamma_n$ winds around the origin at
      least $k_n$ times.
   \end{enumerate}
   \renewcommand{\labelenumi}{(\arabic{enumi})}

   Observe that the length $|AB|$ of any billiard orbit segment, connecting two
   consecutive collisions, is always between $d-2r_0$ and $d$, where $d$ is the
   distance between the centers of the obstacles affected by the
   collisions. Therefore, by the previous lemma we get the following inequality
   for the length $T_n$ of $x_n$:
   \begin{displaymath}
      \frac{T_n}{1-2r_0}+C\ge 4\sqrt{2}k_n
   \end{displaymath}
   with some constant $C>0$. This inequality implies
   \begin{displaymath}
      \sigma(r_0,\, w_\infty)=\lim_{n\to\infty}\frac{4k_n}{T_n}\le
      \frac{\sqrt{2}}{2(1-2r_0)},
   \end{displaymath}
   thus
   \begin{displaymath}
      \limsup_{r_0\to 0}\sigma(r_0,w_\infty)\le\frac{\sqrt{2}}{2},
   \end{displaymath}
   as claimed by the proposition.
\end{proof}

\section{Corollaries and Concluding Remarks}\label{conc_sec}

The first corollary listed in this section is a byproduct of the proof of
Theorem \ref{upper_bnd}. It provides a positive constant as the upper estimate for the
topological entropy $h_{\text{top}}(r_0)$ of our considered $2D$ billiard flow
with one obstacle.

\begin{thm}\label{htop_thm}
   For the topological entropy $h_{\text{top}}(r_0)$ of the billiard
   flow studied in this paper we have the following upper estimate
   \begin{displaymath}
      h_{\text{top}}(r_0)\le 6\sqrt{2}\ln2=5.8815488\dots
   \end{displaymath}
\end{thm}

\begin{rem}
   The above corollary should be compared to (and explained in the
   framework of) some earlier results by Burago-Ferleger-Kononenko. In \cite{BFK98},
   the authors also enumerate all possible homotopical-combinatorial types of
   trajectories, and they prove the existence of a limit
   \begin{displaymath}
      0 < \lim_{r_0 \rightarrow 0} h_{\text{top}} (r_0) = c_0 < \infty.
   \end{displaymath}
   along with the lower estimate $\ln 3 \le c_0$ and an implicit upper bound in
   terms of the similar entropy limit for the $3D$ Lorentz gas. In Theorem \ref{htop_thm} we obtained a concrete upper bound for $c_0$.
\end{rem}

\begin{proof}[Proof of \ref{htop_thm}.]
   We subdivide the periodic billiard table $\mathcal{Q}$ (the configuration space)
   into five pairwise disjoint domains $\mathcal{D}_1, \mathcal{D}_2^\pm,
   \mathcal{D}_3^\pm$ with piecewise linear boundaries as depicted in the figure below.
   The domains $\mathcal{D}_k^+$ ($k=2,3$) consist of all points $(x_1, x_2) \in \mathcal{Q}$ for which the fractional part $\{x_{k-1}\}$ of $x_{k-1}$ satisfies
   the inequality $\{x_{k-1}\}\le\varepsilon_0$ (for some fixed, small $\varepsilon_0
   > 0$), the domains $\mathcal{D}_k^-$ ($k=2,3$) consist of all points $(x_1,
   x_2) \in \mathcal{Q}$ for which $\{x_{k-1}\}\ge 1-\varepsilon_0$, while
   $\mathcal{D}_1$ is the closure $\overline{\mathcal{Q}\setminus(\mathcal{D}_2^-\cup\mathcal{D}_2^+\cup
  \mathcal{D}_3^-\cup\mathcal{D}_3^+)}$ of the set $\mathcal{Q}\setminus(\mathcal{D}_2^-\cup\mathcal{D}_2^+\cup
   \mathcal{D}_3^-\cup\mathcal{D}_3^+)$.
   \begin{figure}[h]
      \centerline{
      \includegraphics[width=0.4\textwidth, height=0.4\textheight]{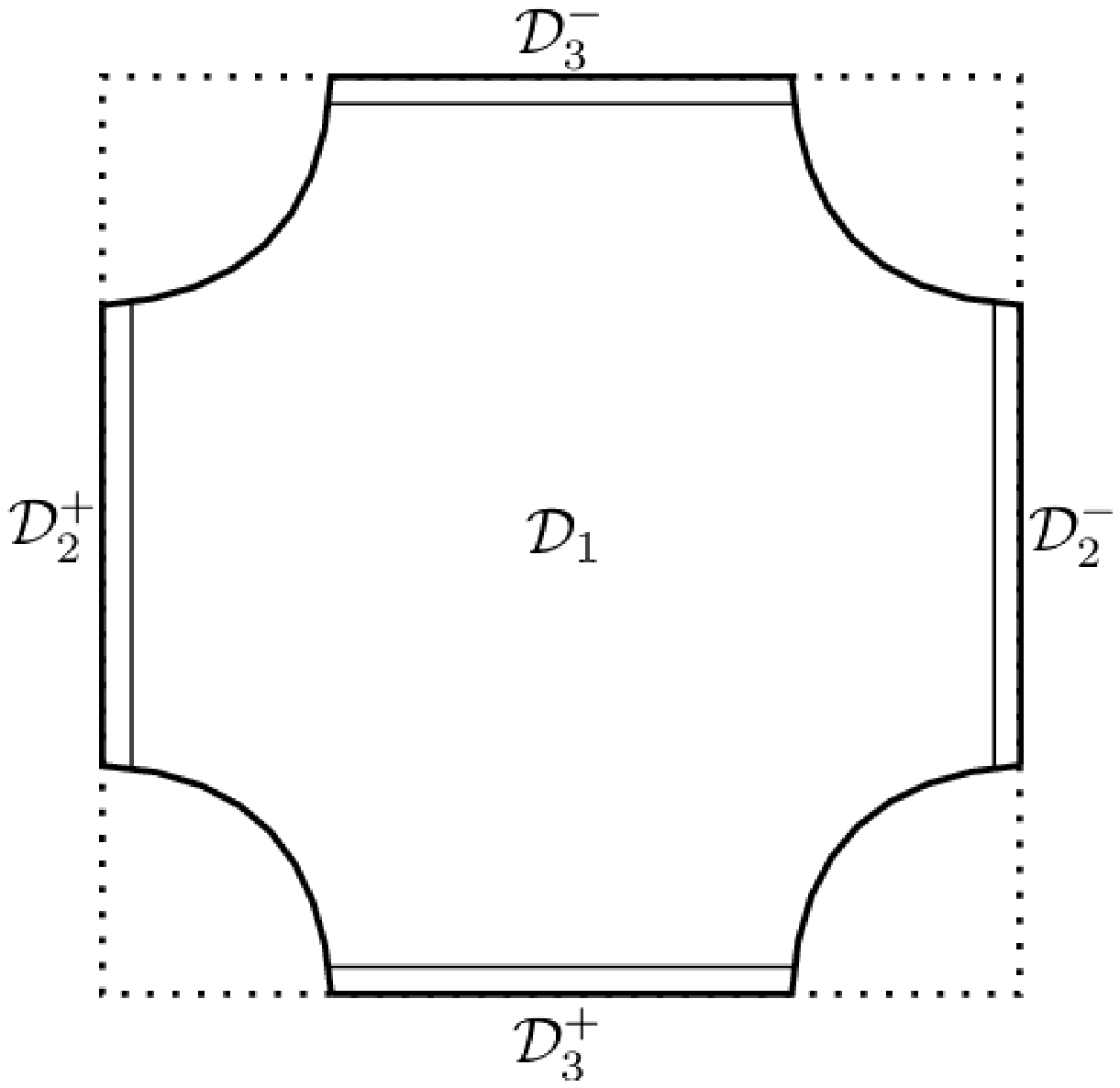}}
      \label{Figure 4}
      \caption{}
   \end{figure}
   The union $\mathcal{Q} =\mathcal{D}_1\cup\mathcal{D}_2^-\cup\mathcal{D}_2^+\cup
   \mathcal{D}_3^-\cup\mathcal{D}_3^+$ is an almost disjoint one: these domains
   only intersect at their piecewise linear boundaries. Thus, from the dynamical
   viewpoint $\mathcal{Q}=\mathcal{D}_1\cup\mathcal{D}_2^-\cup\mathcal{D}_2^+\cup
   \mathcal{D}_3^-\cup\mathcal{D}_3^+$ is a partition $\Pi$.

   We claim that $\Pi$ is a generating partition, meaning that the supremum (the
   coarsest common refinement) $\bigwedge_{n=-\infty}^{\infty}S^{-n\epsilon_0}(\Pi)$ of the
   partitions $S^{-n\epsilon_0}(\Pi)$ is the trivial partition into the singletons, modulo
   the zeroe-measured sets.  Indeed, if two phase points $x = (q_1, v_1)$ and $y
   = (q_2, v_2)$ ($q_1, q_2 \in \mathcal{Q}$, $v_i \in \BR^2$, $\|v_i\| = 1$)
   share the same symbolic future itineraries (recorded at $n\epsilon_0$ moments of time)
   with respect to the partition
   $\Pi$, then, as an elementary inspection shows, $S^{\tau}y$ and $x$ belong to
   the same local stable curve, where $\tau \in \BR$ is a time-synchronizing
   constant. Similar results apply to the shared symbolic itineraries in the past
   and the unstable curves. These facts imply that $x = S^{\tau}y$ (with some
   $\tau \in \BR$), whenever $x$ and $y$ share identical $\Pi$-itineraries in
   both time directions, i.e. $x=y$ for a typical pair $(x,y)$, so $\Pi$ is a generating 
   partition.

   For any time $T > 0$ ($T$ will eventually go to infinity) denote
   by $N(T)$ the number of all possible $\Pi$-itineraries of
   trajectory segments $S^{[0,T]}x$, $x \in \mathcal{M}$. It follows from
   the generating property of $\Pi$ that
   \begin{eqnarray}\label{htop_lim}
      h_{\text{top}} = \lim_{T \rightarrow \infty} \frac{1}{T}\ln N(T).
   \end{eqnarray}
   It is clear that any orbit segment $\left\{x(t)| 0\le t\le T\right\}$
   alternates between the domains $\mathcal{D}_1$ and
   $\mathcal{D}^*=\overline{\mathcal{Q}\setminus\mathcal{D}_1}$.
   Consider an orbit segment $x=\left\{x(t)\big|\; 0\le t\le T\right\}$ lifted to
   the covering space $\tilde{\mathcal{Q}}$ of $\mathcal{Q}$. Let $\tau_1$ be a
   time when $x$ leaves the domain $D_2^+$ ($D_2^-$), and $\tau_2$ be the
   time when $x$ re-enters the same domain $D_2^+$ ($D_2^-$) the next time,
   $0\le\tau_1<\tau_2\le T$. The proof of the lemma following Theorem \ref{upper_bnd}
   shows that
   \begin{displaymath}
      \int_{\tau_1}^{\tau_2}\left|\dot{x}_1(t)\right|dt \ge 1-\varepsilon_0.
   \end{displaymath}
   Therefore, the number of times the orbit segment $x$ visits the domain
   $D_2^+$ ($D_2^-$) is at most
   \begin{displaymath}
      \frac{1}{1-\varepsilon_0}\cdot\int_0^T \left|\dot{x}_1(t)\right|dt+1.
   \end{displaymath}
   Applying this upper estimate to $D_2^+$ and $D_2^-$, then the analogous upper
   estimates for the number of visits to $D_3^\pm$, and, finally, taking the sum
   of the arising four estimates, we get that the total number of visits by $x$
   to the four domains $D_2^\pm$, $D_3^\pm$ is at most
   \begin{displaymath}
      \frac{2}{1-\varepsilon_0}\cdot\int_0^T \left(\left|\dot{x}_1(t)\right|+
      \left|\dot{x}_2(t)\right|\right)dt+4\le\frac{2\sqrt{2}T}{1-\varepsilon_0}+4.
   \end{displaymath}
   Since $x$ alternates between $D_1$ and the union of the other four domains,
   the total number of times $x$ visits $D_1$ is at most
   \begin{displaymath}
      f(T,\,\varepsilon_0):=\frac{2\sqrt{2}T}{1-\varepsilon_0}+5.
   \end{displaymath}
   After entering any of the domains $\mathcal{D}_2^\pm$, $\mathcal{D}_3^\pm$,
   the orbit segment $\sigma$ has two sides of this domain (i. e. two combinatorial
   possibilities) to exit it, whereas, after entering the domain $\mathcal{D}_1$,
   it has four sides of $\mathcal{D}_1$ to leave it. This argument immediately yields
   the upper estimate
   \begin{eqnarray}\label{NT_bnd}
      N(T) \le 8^{f(T,\,\varepsilon_0)}
   \end{eqnarray}
   for the number $N(T)$ of all possible symbolic types of orbit segments of
   length $T$. In light of \eqref{htop_lim}, the above inequality proves the upper
   estimate of Theorem \ref{htop_thm}, once we take the natural logarithm of
   \eqref{NT_bnd}, take the limit as $T\to\infty$, and, finally, the limit as
   $\varepsilon_0\to0$.
\end{proof}

\begin{cor}[Corollary of Theorem \ref{lower_bnd}]
For the topological entropy $h_{\text{top}}(r_0)$ of the billiard flow (with one circular
obstacle of radius $r_0$ inside $\mathbb{T}^2$) we have the lower estimate
\[
h_{\text{top}}(r_0) \geq \frac{\ln 3}{\sqrt{2}} \approx 0.776836199\dots
\]
\end{cor}

\begin{sprf} (A sketch.)
Theorem \ref{lower_bnd} says that the words 
$w\left(\left\{x(t)|\; 0\leq t\leq T\right\}\right)$ corresponding to all orbits
$\left\{x(t)|\; 0\leq t\leq T\right\}$ of length $T$ fill in the ball of radius
$T/\sqrt{2}$ in the Cayley graph of the group $F_2$. Hence the number of different homotopy
types of these orbits $\left\{x(t)|\; 0\leq t\leq T\right\}$ is at least 
$\text{const}\cdot 3^{T/\sqrt{2}}$. Take the natural logarithm of this lower estimate, divide by
$T$, and pass to the limit as $T\to \infty$ to get the claim of the corollary. \qed
\end{sprf}

\medskip

Let $x_{T_0}=\left\{x(t)|\; 0\leq t\leq T_0\right\}$ be a periodic orbit with period $T_0$,
and $w_0=w\left(x_{T_0}\right)\in F_2(a,b)$ the symbolic word corresponding to it. Finally,
let $w_\infty = w_0w_0w_0\dots \in \text{Ends}(F_2)$ be the infinite power of $w_0$.
It is clear that the homotopical rotation number $(t,\, e)=(t,\, w_\infty)\in C$ of the full (periodic) orbit
$x$ exists, i. e.
\[
t=\lim_{T\to \infty} \frac{\left\Vert w\left(\left\{x(t)|\; 0\leq t\leq T\right\}\right)\right\Vert}{T}
= \frac{||w_0||}{T_0},
\]
\[
e=w_\infty =\lim_{T\to \infty} w\left(\left\{x(t)|\; 0\leq t\leq T\right\}\right).
\]
Note that $t=0$ if and only if $w_0=1$. In this case the directional component 
$e=w_\infty \in \text{Ends}(F_2)$ of the rotation number is undefined.

\begin{rem}\label{periodic_rem}
   We observe that in Definition \ref{admis_def} of the admissible homotopical
   rotation set $AR \subset C$ we can select the approximating orbit segments
   $x_i = \{x_i(\tau)|0 \le \tau \le T_i\}$ to be periodic with period $T_i$
   (see Theorem 2.2 of \cite{BMS06}). Thus, the homotopical rotation numbers
   $(t,e) \in C$
   \begin{displaymath}
      e = \lim_{T \rightarrow \infty} w(\{x(\tau)|0 \le \tau \le T\}) \in \text{Ends}(F_2),
   \end{displaymath}
   \begin{displaymath}
      t = \lim_{T \rightarrow \infty} \frac{\|w(\{x(\tau)|0 \le \tau \le
      T\})\|}{T},
   \end{displaymath}
   corresponding to admissible periodic orbits $\{x(\tau)|\tau \in \BR\}$ form
   a dense subset in $AR$.
\end{rem}

\begin{rem}\label{n_obs_rem}
   The problem of defining and thoroughly studying the analogous homotopical
   rotation numbers in the case of $N$ round obstacles in $\mathbb{T}^2$ $(N \ge 2)$
   is much more complex than the case $N=1$. Indeed, the fundamental group
   $G = \pi_1(\mathcal{Q})$ turns out to be the group $\textbf{F}_{N+1}$ freely
   generated by $N+1$ elements $a_1, a_2, \ldots, a_{N+1}$ (see \cite{Mas91}).
   The complexity of the problem is partially explained by the following fact: the
   ``abelianized'' version $G/G^{\prime}$ (where $G^{\prime} = [G,G]$ is the
   commutator subgroup of $G$) is isomorphic to $\BZ^{N+1}$. In the case $N=1$
   the group $\BZ^{N+1} = \BZ^2$ coincides with the lattice group of periodicity of
   the billiard system, and this coincidence establishes a strong connection between
   the newly introduced homotopical (non-commutative) rotation number $(t,e) \in C$
   of a trajectory $\{x(\tau)|\tau \in \BR\}$,
   \begin{displaymath}
      e = \lim_{T \rightarrow \infty} w(\{x(\tau)|0 \le \tau \le T\}) =
      \lim_{T \rightarrow \infty} w_T
   \end{displaymath}
   \begin{displaymath}
      t = \lim_{T \rightarrow \infty} \frac{\|w(\{x(\tau)|0 \le \tau \le
      T\})\|}{T} = \lim_{T \rightarrow \infty} \frac{\|w_T\|}{T}
   \end{displaymath}
   and the traditional (commutative) rotation vector $\rho$ of the same trajectory
   as follows:
   \begin{displaymath}
      \rho = \lim_{T \rightarrow \infty} \frac{1}{T} \pi(w_T) \in \BR^2,
   \end{displaymath}
   where $\pi: G \rightarrow G/G^{\prime} = \BZ^2$ is the natural projection.
   Clearly, there is no such straightforward correspondence between the two types of
   rotation numbers (vectors) in the case $N \ge 2$.
\end{rem}

\begin{rem}\label{arb_obs_rem}
   If one carefully studies all the proofs and arguments of this paper,
   it becomes obvious that the round shape of the sole obstacle $\mathcal{O}$
   was essentially not used. Thus, all the above results carry over to any other
   billiard table model on $\mathbb{T}^2$ with a single strictly convex obstacle with
   smooth boundary $\partial\mathcal{O}$, provided that $\mathcal{O}$ is small in the
   sense of \cite{BMS06}, i.e., $\mathcal{O}$ is contained in a disk of radius $r_0$,
   $r_0 < \sqrt{2}/4$. (And $r_0 < \sqrt{5}/10$ for Theorem \ref{lower_bnd}.)
\end{rem}

\begin{rem}\label{dim_rem}
   One can ask similar questions (regarding the noncommutative rotation numbers/sets)
   for toroidal billiards in the configuration space $\mathcal{Q}$, where \begin{displaymath}
      \mathcal{Q} = \mathbb{T}^d\setminus\bigcup_{i=1}^N \mathcal{O}_i,
   \end{displaymath}
   with $d \ge 3$ and $N$ mutually disjoint, compact, strictly convex obstacles
   $\mathcal{O}_i$ with smooth boundaries. Such a space $\mathcal{Q}$ is, obviously,
   homotopically equivalent to the $d$-torus $\mathbb{T}^d$ with $N$ points removed
   from it (a ``punctured torus''); however, due to the assumption $d \ge 3$, the
   fundamental group $\pi_1(\mathcal{Q})$ of such a space is
   naturally isomorphic to $\pi_1(\mathbb{T}^d) \cong \BZ^d$, for the
   homotopical deformations of loops can always avoid the removed $N$ points.
   Thus, for such a system the homotopical rotation numbers and sets coincide
   with the usual commutative notions, studied in \cite{BMS06}.
\end{rem}
\bibliography{rotation2bib} 
\end{document}